\documentclass[a4paper]{amsproc}
\usepackage{amssymb}
\usepackage{mathrsfs}
\usepackage{amsfonts}
\usepackage{amsthm}
\usepackage{amsmath,amscd}
\usepackage{amsxtra}     
\usepackage{bm}
\usepackage[all,cmtip]{xy}
\usepackage{epsfig}
\usepackage{verbatim}
\usepackage{color}
\usepackage{enumerate}
\usepackage{hyperref}
\usepackage{stmaryrd}
\usepackage{tikz}
\usetikzlibrary{arrows,matrix,shapes,trees}

\newenvironment{claim}[1]{\par\noindent\underline{Claim:}\space#1}{}
\newenvironment{claimproof}[1]{\par\noindent\underline{Proof of Claim:}\space#1}{\hfill $\boxempty$}

\theoremstyle{plain}
 \newtheorem{thm}{Theorem}[section]
 \newtheorem{prop}{Proposition}[section]
 \newtheorem{lem}{Lemma}[section]
 \newtheorem{cor}{Corollary}[section]

\theoremstyle{definition}
 
 \newtheorem{defn}{Definition}[section]
\theoremstyle{remark}
 \newtheorem{rmk}{Remark}[section]
 \numberwithin{equation}{section}

\newcommand{\N}{{\mathbb N}}
\newcommand{\Q}{{\mathbb Q}}
\newcommand{\Z}{{\mathbb Z}}

\newcommand{\Gm}{\mathbb{G}_{\mr{m}}}

\newcommand{\Gml}{\mathbb{G}_{\mr{m,log}}}
\newcommand{\Gmlb}{\overline{\mathbb{G}}_{\mr{m,log}}}

\newcommand{\mr}{\mathrm}
\newcommand{\mc}{\mathcal}

\newcommand{\fsS}{(\mr{fs}/S)}
\newcommand{\fsSet}{(\mr{fs}/S)_{\mr{\acute{e}t}}}
\newcommand{\fsSfl}{(\mr{fs}/S)_{\mr{fl}}}
\newcommand{\fsSket}{(\mr{fs}/S)_{\mr{k\acute{e}t}}}
\newcommand{\fsSkfl}{(\mr{fs}/S)_{\mr{kfl}}}
\newcommand{\Set}{S_{\mr{\acute{e}t}}}
\newcommand{\Sket}{S_{\mr{k\acute{e}t}}}
\newcommand{\Sfl}{S_{\mr{fl}}}
\newcommand{\Skfl}{S_{\mr{kfl}}}
\newcommand{\Spec}{\mathop{\mr{Spec}}}

\title[Extending finite group schemes via log abelian varieties]{Extending finite subgroup schemes of semi-stable abelian varieties via log abelian varieties}



\author[Heer Zhao]{\bfseries Heer Zhao}

\address{
    Heer Zhao, 
    Fakult\"at f\"ur Mathematik, 
    Universit\"at Duisburg-Essen, 
    Essen 45117, 
    Germany, 
    	heer.zhao@gmail.com}

\begin{document}

\vspace{18mm} \setcounter{page}{1} \thispagestyle{empty}

\begin{abstract}
For a semi-stable abelian variety $A_K$ over a complete discrete valuation field $K$, we show that every finite subgroup scheme of $A_K$ extends to a log finite flat group scheme over the valuation ring of $K$ endowed with the canonical log structure. To achieve this, we first prove that every weak log abelian variety over an fs log scheme with its underlying scheme locally noetherian, is a sheaf for the Kummer flat topology, which answers a question of Chikara Nakayama. We also give several equivalent conditions defining isogenies of log abelian varieties. 
\end{abstract}

\maketitle

\section*{Introduction}
As stated in \cite{k-k-n2}, degenerating abelian varieties can not preserve group structure, properness, and smoothness at the same time. Log abelian variety is a construction aimed to make the impossible possible in the world of log geometry. The idea dates back to Kazuya Kato's construction of log Tate curve in \cite[Sec. 2.2]{kat3}, in which he also conjectured the existence of a general theory of log abelian varieties. The theory finally comes true in \cite{k-k-n1} and \cite{k-k-n2}, and has been further developed in \cite{k-k-n3} and \cite{k-k-n4}. 

Log abelian varieties are defined as certain sheaves of abelian groups on $\fsSet$, see \cite[Def. 4.1]{k-k-n2}. Weak log abelian varieties are defined in \cite[1.6]{k-k-n4} as generalisations of log abelian varieties, and they are also sheaves of abelian groups on $\fsSet$. It is natural to expect that weak log abelian varieties are also sheaves for the classical flat topology, the Kummer \'etale topology, and the Kummer flat topology. In fact, the case for the Kummer \'etale topology has been proven, see \cite[Thm. 5.1]{k-k-n4}. Weak log abelian varieties with constant degeneration (see \cite[1.7]{k-k-n4}) are special cases of weak log abelian varieties, and they are sheaves for the Kummer flat topology by \cite[Thm. 2.1 (1)]{zha3}. In the first section, we prove that weak log abelian varieties are sheaves for the Kummer flat topology, see Theorem \ref{thm}. 

After knowing that log abelian varieties are sheaves for the Kummer flat topology, we study isogenies of log abelian varieties in the second section. We give several equivalent descriptions of isogeny, see Proposition \ref{prop2.1}. 

As mentioned in the beginning of \cite{kat3}, a finite subgroup scheme of a semi-stable abelian variety $A_K$ over a complete discrete valuation field $K$, does not necessarily extend to a finite flat group scheme over the corresponding valuation ring $R$. Again we can make the impossible possible in the world of log geometry. In the last section, we show that every finite subgroup scheme of a semi-stable abelian variety over $K$ can always be extended to a log finite flat group scheme over $S$, where $S$ denotes $\Spec R$ endowed with the canonical log structure, see Theorem \ref{thm3.2}. A key tool is an equivalence of categories between semi-stable abelian varieties over $K$ and log abelian varieties over $S$ due to Takeshi Kajiwara, Kazuya Kato and Chikara Nakayama, see \cite[13.4]{k-k-n4}.

\section*{Notation and conventions}
Let $S$ be an fs log scheme with its underlying scheme locally noetherian, we denote by $(\mr{fs}/S)$ the category of fs log schemes over $S$, and denote by $(\mr{fs}/S)_{\mr{\acute{e}t}}$ (resp. $(\mr{fs}/S)_{\mr{k\acute{e}t}}$, resp. $(\mr{fs}/S)_{\mr{fl}}$, resp. $(\mr{fs}/S)_{\mr{kfl}}$) the classical \'etale site (resp. Kummer \'etale site, resp. classical flat site, resp. Kummer flat site) on $(\mr{fs}/S)$. In order to shorten formulas, we will mostly abbreviate $\fsSet$ (resp. $\fsSket$, resp. $\fsSfl$, resp. $\fsSkfl$) as $\Set$ (resp. $\Sket$, resp. $\Sfl$, resp. $\Skfl$). We refer to \cite[2.5]{ill1} for the classical \'etale site and the Kummer \'etale site, and \cite[Def. 2.3]{kat2} and \cite[\S 2.1]{niz1} for the Kummer flat site. The definition of the classical flat site is an obvious analogue of that of the classical \'etale site. Then we have two natural ``forgetful'' map of sites:
\begin{equation}\label{eq0.1}
\varepsilon_{\mr{\acute{e}t}}:(\mr{fs}/S)_{\mr{k\acute{e}t}}\rightarrow (\mr{fs}/S)_{\mr{\acute{e}t}}
\end{equation}
and
\begin{equation}\label{eq0.2}
\varepsilon_{\mr{fl}}:(\mr{fs}/S)_{\mr{kfl}}\rightarrow (\mr{fs}/S)_{\mr{fl}} .
\end{equation}
We also have a natural map of sites
\begin{equation}\label{eq0.3}
m:(\mr{fs}/S)_{\mr{fl}}\rightarrow (\mr{fs}/S)_{\mr{\acute{e}t}} .
\end{equation}
Let 
\begin{equation}\label{eq0.4}
\delta:(\mr{fs}/S)_{\mr{kfl}}\rightarrow (\mr{fs}/S)_{\mr{\acute{e}t}} 
\end{equation}
be the composition of $m$ and $\varepsilon_{\mr{fl}}$.

Kato's multiplicative group (or the log multiplicative group) $\Gml$ is the sheaf on $\Set$ defined by $\Gml(U)=\Gamma(U,M^{\mr{gp}}_U)$
for any $U\in\fsS$, where $M_U$ denotes the log structure of $U$ and $M^{\mr{gp}}_U$ denotes the group envelope of $M_U$. The Kummer \'etale sheaf $\Gml$ is also a sheaf on $\Skfl$, see \cite[Cor. 2.22]{niz1} for a proof.

By convention, for any sheaf of abelian groups $F$ on $\Skfl$ and a subgroup sheaf $G$ of $F$ on $\Skfl$, we denote by $(F/G)_{\Set}$ the quotient sheaf on $\Set$, while $F/G$ denotes the quotient sheaf on $\Skfl$. We abbreviate the quotient sheaf $\Gml/\Gm$ on $\Skfl$ as $\Gmlb$.

\section{Weak log abelian varieties are kfl sheaves}
Let $A$ be a weak log abelian variety over $S$. The main result of this section is the following theorem.

\begin{thm}\label{thm}
The weak log abelian variety $A$ is a sheaf with respect to the Kummer flat topology.
\end{thm}

Firstly by the definition in \cite[1.6]{k-k-n4}, \'etale locally on $S$, there is a semi-abelian scheme $G$ over $S$, finitely generated free $\Z$-modules $X$ and $Y$, an admissible pairing \begin{equation}\label{eq1.1}
\langle,\rangle:X\times Y\rightarrow (\Gml/\Gm)_{\Set}
\end{equation}
on $\Set=\fsSet$, and an exact sequence 
\begin{equation}\label{eq1.2}
0\rightarrow G\rightarrow A\rightarrow \mc{Q}_{\mr{\acute{e}t}}/\overline{Y}\rightarrow 0 
\end{equation}
on $S_{\mr{\acute{e}t}}$ with $\mc{Q}_{\mr{\acute{e}t}}:=\mc{H}om_{\Set}(X,(\Gml/\Gm)_{\Set})^{(Y)}$ and $\overline{Y}$ the image of $Y$ inside $\mc{Q}_{\mr{\acute{e}t}}$. As explained in \cite[9.2, 9.3]{k-k-n2} and \cite[1.7]{k-k-n4}, $G$ and the pairing $\overline{X}\times\overline{Y}\rightarrow (\Gml/\Gm)_{\mr{et}}$ induced from (\ref{eq1.1}) exist globally. We call $G$ the semi-abelian part of $A$, and $\mc{Q}_{\mr{\acute{e}t}}/\overline{Y}$ the discrete part of $A$.

For the proof of Theorem \ref{thm}, we may assume that we are in the situation in \cite[2.4]{k-k-n4}, i.e. the pairing $\langle,\rangle$ exists globally. Let $\tilde{A}$ be as defined in \cite[2.4]{k-k-n4}, then we have a commutative diagram  
\begin{equation}\label{eq1.3}
\xymatrix{
&&0\ar[d] &0\ar[d]   \\
&&\overline{Y}\ar@{=}[r]\ar[d] &\overline{Y}\ar[d]  \\
0\ar[r] &G\ar[r]\ar@{=}[d] &\tilde{A}\ar[r]\ar[d] &\mc{Q}_{\mr{\acute{e}t}}\ar[r]\ar[d] &0  \\
0\ar[r] &G\ar[r] &A\ar[r]\ar[d] &\mc{Q}_{\mr{\acute{e}t}}/\overline{Y}\ar[r]\ar[d] &0  \\
&&0 &0
}
\end{equation}
on $S_{\mr{\acute{e}t}}$ with exact rows and columns. 

\begin{lem}\label{lem1}
We have that $\tilde{A}$ is a sheaf for the Kummer flat topology.
\end{lem}
\begin{proof}
The argument is almost identical to the part of \cite[5.3]{k-k-n4}, which shows that $\tilde{A}$ is a sheaf for the Kummer \'etale topology.

To show $\tilde{A}$ is a sheaf for the Kummer flat topology, it amounts to showing that the canonical homomorphism $\tilde{A}\rightarrow \delta_*\delta^{-1}\tilde{A}$ coming from the adjunction $(\delta^{-1},\delta_*)$ is an isomorphism.

Applying the adjunction $(\delta^{-1},\delta_*)$ to the short exact sequence
$$0\rightarrow G\rightarrow\tilde{A}\rightarrow\mc{Q}_{\mr{\acute{e}t}}\rightarrow 0,$$
we get the following commutative diagram
$$
\xymatrix{
0\ar[r] &G\ar[r]\ar[d]^{\alpha} &\tilde{A}\ar[r]\ar[d]^{\beta} &\mc{Q}_{\mr{\acute{e}t}}\ar[r]\ar[d]^{\gamma} &0 \\
0\ar[r] &\delta_*\delta^{-1}G\ar[r] &\delta_*\delta^{-1}\tilde{A}\ar[r]^{\theta} &\delta_*\delta^{-1}\mc{Q}_{\mr{\acute{e}t}}
}$$
with exact rows. Since $G$ is a group scheme, $\alpha$ is an isomorphism by \cite[Thm. 5.2]{k-k-n4}. By \cite[Lem. 2.3]{zha3}, $\gamma$ is identical to the canonical injection $\mc{Q}_{\mr{\acute{e}t}}\rightarrow \mc{Q}_{\mr{\acute{e}t}}\otimes_{\Z}\Q$. It follows that $\beta$ is injective. We are reduced to show the surjectivity of $\beta$. It suffices to show that the image of $\theta$ lands in the image of $\gamma$. Since the surjectivity of a morphism of \'etale sheaves can be checked stalkwise, and the formations of $\mc{Q}_{\mr{\acute{e}t}}$ and $\mc{Q}_{\mr{\acute{e}t}}\otimes_{\Z}\Q$ are compatible with strict base change, we may assume that the base is a log point. Then we are reduced to show the surjectivity of $\beta$ for the case that $A$ has constant reduction. Moreover we may assume that we are in the situation of \cite[2.5]{k-k-n4}. Then by \cite[Thm. 4.1]{k-k-n4}, the $\Delta$-part $\tilde{A}^{(\Delta)}$ of $\tilde{A}$ is represented by an fs log scheme for any finitely generated subcone $\Delta$ of $C$ (for the notation $C$ and $\tilde{A}^{(\Delta)}$, see \cite[Thm. 4.1]{k-k-n4}). Hence $\tilde{A}^{(\Delta)}$ is a sheaf for the Kummer flat topology by \cite[Thm. 5.2]{k-k-n4}. Since $\tilde{A}$ is the union of $\tilde{A}^{(\Delta)}$ with $\Delta$ running over the set of all finitely generated subcones of $C$, it is a sheaf for the Kummer flat topology. In particular $\beta:\tilde{A}\rightarrow \delta_*\delta^{-1}\tilde{A}$ is surjective.
\end{proof}

Since $\tilde{A}$ is a sheaf with respect to the Kummer flat topology, we can identify $\delta_*\delta^{-1}\tilde{A}$ with $\tilde{A}$ canonically. 

\begin{lem}\label{lem2}
The sheaf  $\overline{Y}$ on $S_{\mr{\acute{e}t}}$ is also a sheaf for the Kummer flat topology.
\end{lem}
\begin{proof}
It suffices to show that the canonical map $\overline{Y}\rightarrow\delta_*\delta^{-1}\overline{Y}$ is an isomorphism. This is a local problem, and we may assume that the underlying scheme of $S$ is noetherian. Since $\overline{Y}$ is a constructible sheaf on $S_{\mr{\acute{e}t}}$, there exists a shortest chain of closed subschemes 
\begin{equation}\label{eq.const}
\emptyset=S_0\subsetneq S_1 \subsetneq \cdots \subsetneq S_n=S
\end{equation}
such that the restriction of $\overline{Y}$ to $U_i:=S_i-S_{i-1}$ is locally constant for $i=1,\cdots, n$. 

We use induction on $n$ now. 

If $n=1$, then $\overline{Y}$ is locally constant. Let $f:S'\rightarrow S$ be a strict \'etale cover, i.e. $f$ is an \'etale cover of schemes and the log structure of $S'$ is induced from that of $S$, such that the restriction of $\overline{Y}$ to $S'$ is constant. And let $f_{\mr{\acute{e}t}}:(\mr{fs}/S')_{\mr{\acute{e}t}}\rightarrow (\mr{fs}/S)_{\mr{\acute{e}t}}$ be the map of sites induced by $f$. To show that the canonical map $\overline{Y}\rightarrow\delta_*\delta^{-1}\overline{Y}$ is an isomorphism, it suffices to show that its restriction $f_{\mr{\acute{e}t}}^{-1}\overline{Y}\rightarrow f_{\mr{\acute{e}t}}^{-1}\delta_*\delta^{-1}\overline{Y}$ to $(\mr{fs}/S')_{\mr{\acute{e}t}}$ is an isomorphism. By an easy computation, the map $f_{\mr{\acute{e}t}}^{-1}\overline{Y}\rightarrow f_{\mr{\acute{e}t}}^{-1}\delta_*\delta^{-1}\overline{Y}$ can be identified with the canonical map $f_{\mr{\acute{e}t}}^{-1}\overline{Y}\rightarrow\delta_*\delta^{-1}f_{\mr{\acute{e}t}}^{-1}\overline{Y}$. The constant sheaf $f_{\mr{\acute{e}t}}^{-1}\overline{Y}$ is representable by a group scheme, hence $f_{\mr{\acute{e}t}}^{-1}\overline{Y}\rightarrow\delta_*\delta^{-1}f_{\mr{\acute{e}t}}^{-1}\overline{Y}$ is an isomorphism by \cite[Thm. 5.2]{k-k-n4}.

Now we assume $n>1$. Let $j_{\mr{\acute{e}t}}$ (resp. $j_{\mr{kfl}}$) be the canonical map $(\mr{fs}/U_n)_{\mr{\acute{e}t}}\rightarrow (\mr{fs}/S_{n})_{\mr{\acute{e}t}}$ (resp. $(\mr{fs}/U_n)_{\mr{kfl}}\rightarrow (\mr{fs}/S_{n})_{\mr{kfl}}$) induced by the open immersion $U_n\rightarrow S_n$, and let $i_{\mr{\acute{e}t}}$ (resp. $i_{\mr{kfl}}$) be the canonical map $(\mr{fs}/S_{n-1})_{\mr{\acute{e}t}}\rightarrow (\mr{fs}/S_{n})_{\mr{\acute{e}t}}$ (resp. $(\mr{fs}/S_{n-1})_{\mr{kfl}}\rightarrow (\mr{fs}/S_{n})_{\mr{kfl}}$) induced by the closed immersion $S_{n-1}\rightarrow S_n$. By induction the \'etale sheaf $i_{\mr{\acute{e}t}}^{-1}\overline{Y}$ is also a sheaf on $(\mr{fs}/S_{n-1})_{\mr{kfl}}$. We have a short exact sequence
$$0\rightarrow j_{\mr{\acute{e}t}!}j_{\mr{\acute{e}t}}^{-1}\overline{Y}\rightarrow \overline{Y}\rightarrow i_{\mr{\acute{e}t}*}i_{\mr{\acute{e}t}}^{-1}\overline{Y}\rightarrow 0.$$
For the existence of the left adjoint functor $j_{\mr{\acute{e}t}!}$ (resp. $j_{\mr{kfl}!}$) of $j_{\mr{\acute{e}t}}^{-1}$ (resp. $j_{\mr{kfl}}^{-1}$), we refer to \cite[\href{http://stacks.math.columbia.edu/tag/00XZ}{Tag 00XZ}]{stacks-project}.
Applying the adjoint functors $\delta^{-1}$ and $\delta_*$ to the above short exact sequence, we get the following commutative diagram 
\begin{equation}\label{eq1.4}
\xymatrix{
0\ar[r] &j_{\mr{\acute{e}t}!}j_{\mr{\acute{e}t}}^{-1}\overline{Y}\ar[r]\ar[d]_{\alpha} &\overline{Y}\ar[r]\ar[d] &i_{\mr{\acute{e}t}*}i_{\mr{\acute{e}t}}^{-1}\overline{Y}\ar[r]\ar[d] &0 \\
0\ar[r] &\delta_*\delta^{-1}j_{\mr{\acute{e}t}!}j_{\mr{\acute{e}t}}^{-1}\overline{Y}\ar[r] &\delta_*\delta^{-1}\overline{Y}\ar[r] &\delta_*\delta^{-1}i_{\mr{\acute{e}t}*}i_{\mr{\acute{e}t}}^{-1}\overline{Y}
}
\end{equation} with exact rows, where the vertical maps are given by the adjunction $(\delta^{-1},\delta_*)$. To show that the vertical map in the middle is an isomorphism, we are reduced to show that the left and the right vertical maps are both isomorphisms. 

\begin{claim}
Both the domain and the target of $\alpha$ are supported over $U_n$.
\end{claim}
\begin{claimproof}
The domain of $\alpha$ is clearly supported over $U_n$. We show that the target of $\alpha$ is also supported over $U_n$. For simplification of notation, we denote $j_{\mr{\acute{e}t}!}j_{\mr{\acute{e}t}}^{-1}\overline{Y}$ by $F$. It suffices to show that for each $U\in(\mr{fs}/S)$, the restriction of $\delta_*\delta^{-1}F$ to the small \'etale site of $U$ has trivial stalk at any $u\in U$ which does not lie over $U_n$. Let $x$ be an element of the stalk at $u$, and suppose that $x$ is represented by a section $s\in \delta^{-1}F(V)$ for some \'etale neighborhood $V\rightarrow U,v\mapsto u$ of $u$. By \cite[\href{http://stacks.math.columbia.edu/tag/00WK}{Tag 00WK}]{stacks-project}, there exists a Kummer flat cover $p:V'\rightarrow V$, such that $s$ is represented by a section $s'$ of $F(V')$ which satisfies that $p_1^*(s')$ and $p_2^*(s')$ agree over some Kummer flat cover $V''$ of $V'\times_VV'$, where $p_1$ (resp. $p_2$) is the projection from $V'\times_VV'$ to its first (resp. second) factor. We are going to shrink $V$ and $V'$ so that $p$ remains a Kummer flat cover and $s'$ becomes zero after shrinking.

Let $v'\in V'$ be a preimage of $v\in V$. Since $V'$ is clearly not over $U_n$, the section $s'$ becomes zero after restricting to some open neighborhood $W'$ of $v'$ in $V'$. Since the underlying map of schemes of the Kummer flat cover $V'\rightarrow V$ is open by Lemma \ref{lemA.1}, $W:=p(W')$ is an open subscheme of $V$. We regard $W$ (resp. $W'$) as a log scheme with the induced log structure from $V$ (resp. $V'$). Since log flat morphisms and Kummer morphisms are stable under compositions, the morphism $p\mid_{W'}:W'\rightarrow W$ is both log flat and Kummer. By \cite[\href{http://stacks.math.columbia.edu/tag/01TQ}{Tag 01TQ}]{stacks-project}, the morphism $p\mid_{W'}:W'\rightarrow W$ is locally of finite presentation. It follows that $p\mid_{W'}:W'\rightarrow W$ is a Kummer flat cover. Note that $W$ is also an \'etale neighborhood of $u\in U$. It follows that $x$ can also be represented by $s'\mid_{W'}=0$, hence $x=0$. This finishes the proof of the claim.  
\end{claimproof}

Now we get back to the proof of Lemma \ref{lem2}. By the above claim, to show that $\alpha$ is an isomorphism,  it suffices to check that the restriction of $\alpha$ to $U_n$ is an isomorphism. The map
$$j_{\mr{\acute{e}t}}^{-1}(\alpha):j_{\mr{\acute{e}t}}^{-1}(j_{\mr{\acute{e}t}!}j_{\mr{\acute{e}t}}^{-1}\overline{Y})\rightarrow j_{\mr{\acute{e}t}}^{-1}(\delta_*\delta^{-1}j_{\mr{\acute{e}t}!}j_{\mr{\acute{e}t}}^{-1}\overline{Y})$$
can be identified with the canonical map 
$$j_{\mr{\acute{e}t}}^{-1}(j_{\mr{\acute{e}t}!}j_{\mr{\acute{e}t}}^{-1}\overline{Y})\rightarrow \delta_*\delta^{-1}j_{\mr{\acute{e}t}}^{-1}(j_{\mr{\acute{e}t}!}j_{\mr{\acute{e}t}}^{-1}\overline{Y}).$$
By \cite[\href{http://stacks.math.columbia.edu/tag/00Y2}{Tag 00Y2}]{stacks-project}, the sheaf $j_{\mr{\acute{e}t}}^{-1}(j_{\mr{\acute{e}t}!}j_{\mr{\acute{e}t}}^{-1}\overline{Y})=(j_{\mr{\acute{e}t}}^{-1}j_{\mr{\acute{e}t}!})j_{\mr{\acute{e}t}}^{-1}\overline{Y}$ is canonically identified with the sheaf $j_{\mr{\acute{e}t}}^{-1}\overline{Y}$. Hence the map $j_{\mr{\acute{e}t}}^{-1}(\alpha)$ can be further identified with the canonical map $j_{\mr{\acute{e}t}}^{-1}\overline{Y}\rightarrow \delta_*\delta^{-1}j_{\mr{\acute{e}t}}^{-1}\overline{Y}$. The later is an isomorphism by the case $n=1$. It follows that the left vertical map of (\ref{eq1.4}) is an isomorphism. 

Since $i_{\mr{\acute{e}t}}^{-1}\overline{Y}$ is a sheaf for the Kummer flat topology, so is the \'etale sheaf $i_{\mr{\acute{e}t}*}i_{\mr{\acute{e}t}}^{-1}\overline{Y}$. Hence the right vertical map of (\ref{eq1.4}) is an isomorphism.
\end{proof}

\begin{cor}\label{cor}
Let the notation be as in the proof of Lemma \ref{lem2}, then we have a canonical short exact sequence
$$0\rightarrow j_{\mr{kfl}!}j_{\mr{kfl}}^{-1}\overline{Y}\rightarrow \overline{Y}\rightarrow i_{\mr{kfl}*}i_{\mr{kfl}}^{-1}\overline{Y}\rightarrow 0$$
of sheaves of abelian groups on $S_{\mr{kfl}}$.
\end{cor}
\begin{proof}
By the proof of Lemma \ref{lem2}, we have a short exact sequence $$0\rightarrow j_{\mr{\acute{e}t}!}j_{\mr{\acute{e}t}}^{-1}\overline{Y}\rightarrow \overline{Y}\rightarrow i_{\mr{\acute{e}t}*}i_{\mr{\acute{e}t}}^{-1}\overline{Y}\rightarrow 0$$
of sheaves of abelian groups on $(\mr{fs}/S)_{\mr{kfl}}$. Since we have $j_{\mr{\acute{e}t}!}j_{\mr{\acute{e}t}}^{-1}\overline{Y}=j_{\mr{kfl}!}j_{\mr{kfl}}^{-1}\overline{Y}$ and $i_{\mr{\acute{e}t}*}i_{\mr{\acute{e}t}}^{-1}\overline{Y}=i_{\mr{kfl}*}i_{\mr{kfl}}^{-1}\overline{Y}$, the result follows.
\end{proof}

Applying the adjoint functors $\delta^{-1}$ and $\delta_*$ to the short exact sequence $$0\rightarrow \overline{Y}\rightarrow \tilde{A}\rightarrow A\rightarrow 0,$$
we get the following commutative diagram
\begin{equation}\label{eq1.5}
\xymatrix{
0\ar[r] &\overline{Y}\ar[r]\ar@{=}[d] &\tilde{A}\ar[r]\ar@{=}[d] &A\ar[r]\ar[d] &0 \\
0\ar[r] &\overline{Y}\ar[r] &\tilde{A}\ar[r] &\delta_*\delta^{-1}A\ar[r] &R^1\delta_*\overline{Y}
}
\end{equation}
with exact rows, where the vertical maps are given by the adjunction $(\delta^{-1},\delta_*)$. In order to prove Theorem \ref{thm}, it is enough to show $R^1\delta_*\overline{Y}=0$. 

\begin{lem}\label{lem3}
The sheaf $R^1\delta_*\overline{Y}$ is zero.
\end{lem}
\begin{proof}
We proceed by induction on $n$, with $n$ as in (\ref{eq.const}). Let the notation be as in the proof of Lemma \ref{lem2}. The short exact sequence
$$0\rightarrow j_{\mr{kfl}!}j_{\mr{kfl}}^{-1}\overline{Y}\rightarrow \overline{Y}\rightarrow i_{\mr{kfl}*}i_{\mr{kfl}}^{-1}\overline{Y}\rightarrow 0$$
induces an exact sequence
\begin{equation}\label{eq1.6}
R^1\delta_*j_{\mr{kfl}!}j_{\mr{kfl}}^{-1}\overline{Y}\rightarrow R^1\delta_*\overline{Y}\rightarrow R^1\delta_*i_{\mr{kfl}*}i_{\mr{kfl}}^{-1}\overline{Y}.
\end{equation}
In the proof of Lemma \ref{lem2}, we have seen that the sheaf $j_{\mr{kfl}!}j_{\mr{kfl}}^{-1}\overline{Y}=j_{\mr{\acute{e}t}!}j_{\mr{\acute{e}t}}^{-1}\overline{Y}$ is represented by an \'etale group scheme over $S$. Then the argument from \cite[Lem. 2.4]{zha3} also works for $j_{\mr{kfl}!}j_{\mr{kfl}}^{-1}\overline{Y}$, hence we have $R^1\delta_*j_{\mr{kfl}!}j_{\mr{kfl}}^{-1}\overline{Y}=0$. We are left with proving that $R^1\delta_* i_{\mr{kfl}*}i_{\mr{kfl}}^{-1}\overline{Y}$ is zero.

The Leray spectral sequences for $\delta_* i_{\mr{kfl}*}$ and $i_{\mr{\acute{e}t}*}\delta_*$ give us the following diagram
$$\xymatrix{
0\ar[r] &R^1\delta_* i_{\mr{kfl}*}i_{\mr{kfl}}^{-1}\overline{Y}\ar[r]  &R^1(\delta i_{\mr{kfl}})_*i_{\mr{kfl}}^{-1}\overline{Y}\ar@{=}[d]  \\
0\ar[r] &R^1i_{\mr{\acute{e}t}*} \delta_*i_{\mr{kfl}}^{-1}\overline{Y}\ar[r]  &R^1(i_{\mr{\acute{e}t}} \delta)_*i_{\mr{kfl}}^{-1}\overline{Y}\ar[r] &i_{\mr{\acute{e}t}*} R^1\delta_*i_{\mr{kfl}}^{-1}\overline{Y} 
}$$
with exact rows. Since $i$ is a closed immersion, the functor $i_{\mr{\acute{e}t}*}$ has trivial higher derived functors. The sheaf $R^1\delta_* i_{\mr{kfl}}^{-1}\overline{Y}$ is zero by the induction hypothesis, so is the sheaf $R^1(i_{\mr{\acute{e}t}} \delta)_*i_{\mr{kfl}}^{-1}\overline{Y}$. It follows that the sheaf $R^1\delta_* i_{\mr{kfl}*}i_{\mr{kfl}}^{-1}\overline{Y}$ is zero.
\end{proof}

By Lemma \ref{lem3} and the diagram (\ref{eq1.5}), we conclude Theorem \ref{thm}.

Since $A$ is a sheaf on $\Skfl$ by Theorem \ref{thm}, we can ask what the quotient of $A$ by $G$ on $\Skfl$ is, or equivalently, how the pull-back of the short exact sequence (\ref{eq1.2}) to $\Skfl$ looks like.

\begin{prop}
Let $\mc{Q}$ be $\mc{H}om_{\Skfl}(X,\Gmlb)^{(Y)}$. Then the pullback of the sheaf $\mc{Q}_{\mr{\acute{e}t}}/\overline{Y}$ on $\Set$ to $\Skfl$ is canonically identified with $\mc{Q}/\overline{Y}$. Hence we have a canonical exact sequence 
\begin{equation}\label{eq1.7}
0\rightarrow G\rightarrow A\rightarrow \mc{Q}/\overline{Y}\rightarrow 0 
\end{equation}
on $\Skfl$. We also call $\mc{Q}/\overline{Y}$ the discrete part of $A$.
\end{prop}
\begin{proof}
By \cite[Thm. 5.2]{k-k-n4}, $G$ is a sheaf on $S_{\mr{kfl}}$. By \cite[Lem. 2.3]{zha3}, we have $\delta^{-1}\mc{Q}_{\mr{\acute{e}t}}=\mc{Q}$.  By Lemma \ref{lem1} (resp. Lemma \ref{lem2}, resp. Theorem \ref{thm}), $\tilde{A}$ (resp. $\overline{Y}$, resp. $A$) is a sheaf on $S_{\mr{kfl}}$. Hence the pull-back of (\ref{eq1.3}) to $S_{\mr{kfl}}$ gives rise to a commutative diagram
\begin{equation}\label{eq1.9}
\xymatrix{
&&0\ar[d] &0\ar[d]   \\
&&\overline{Y}\ar@{=}[r]\ar[d] &\overline{Y}\ar[d]  \\
0\ar[r] &G\ar[r]\ar@{=}[d] &\tilde{A}\ar[r]\ar[d] &\mc{Q}\ar[r]\ar[d] &0  \\
0\ar[r] &G\ar[r] &A\ar[r]\ar[d] &\delta^{-1}(\mc{Q}_{\mr{\acute{e}t}}/\overline{Y})\ar[r]\ar[d] &0  \\
&&0 &0
}
\end{equation}
with exact rows and columns. It follows that $\delta^{-1}(\mc{Q}_{\mr{\acute{e}t}}/\overline{Y})=\mc{Q}/\overline{Y}$, and we have a short exact sequence $0\rightarrow G\rightarrow A\rightarrow \mc{Q}/\overline{Y}\rightarrow 0$ on $S_{\mr{kfl}}$.
\end{proof}

\section{Isogenies of log abelian varieties}

First of all, we recall the definitions of certain finite group objects on the site $S_{\mr{kfl}}$. They are needed for defining isogenies of log abelian varieties, in the same way as finite flat group schemes are needed for defining isogenies of abelian schemes. The theory of log finite flat group schemes is due to Kazuya Kato, and we refer to \cite[Appendix]{zha3} for a brief summary and detailed references.

\begin{defn}\label{defn2.1}
The category $(\mr{fin}/S)_{\mr{c}}$ is the full subcategory of the category of sheaves of finite abelian groups on $S_{\mr{kfl}}$ consisting of objects which are representable by a classical finite flat group scheme over $S$. Here ``classical'' means that the log structure of the representing log scheme is the one induced from $S$. 

The category $(\mr{fin}/S)_{\mr{f}}$ is the full subcategory of the category of sheaves of finite abelian groups on $S_{\mr{kfl}}$ consisting of objects which are representable by a classical finite flat group scheme over a log flat cover of $S$. Let $F\in (\mr{fin}/S)_{\mr{f}}$, let $U\rightarrow S$ be a log flat cover of $S$ such that $F_U:=F\times_S U\in (\mr{fin}/S)_{\mr{c}}$, the rank of $F$ is defined to be  the rank of $F_U$ over $U$.

The category $(\mr{fin}/S)_{\mr{r}}$ is the full subcategory of $(\mr{fin}/S)_{\mr{f}}$ consisting of objects which are representable by a log scheme over $S$. We call an object of $(\mr{fin}/S)_{\mr{r}}$ a log finite flat group scheme over $S$.

The category $(\mr{fin}/S)_{\mr{d}}$ is the full subcategory of $(\mr{fin}/S)_{\mr{r}}$ consisting of objects whose Cartier duals also lie in $(\mr{fin}/S)_{\mr{r}}$.
\end{defn}

Now we are ready to define isogenies of log abelian varieties.

\begin{defn}\label{def2.1}
Let $f:A\rightarrow A'$ be a homomorphism of log abelian varieties over $S$. We call $f$ an isogeny of log abelian varieties, if $\mr{ker}f\in (\mr{fin}/S)_{\mr{r}}$ and $f$ is surjective.
\end{defn}

Isogenies of abelian varieties can be described by several equivalent conditions. We will give equivalent descriptions of isogenies of log abelian varieties. We need a lemma first.

\begin{lem}\label{lem2.1}
Let $f:A\rightarrow A'$ be a homomorphism of log abelian varieties over $S$. Let $G$ (resp. $G'$) be the semi-abelian part of $A$ (resp. $A'$), $\mc{Q}_{\mr{\acute{e}t}}/\overline{Y}$ (resp. $\mc{Q}'_{\mr{\acute{e}t}}/\overline{Y}'$) the discrete part of $A$ (resp. $A'$) on $\Set$, and $\mc{Q}/\overline{Y}$ (resp. $\mc{Q}'/\overline{Y}'$) the discrete part of $A$ (resp. $A'$) on $\Skfl$. Then we have a canonical commutative diagram
\begin{equation}\label{eq2.1}
\xymatrix{
0\ar[r] &G\ar[r]\ar[d]^{f_{\mr{c}}} &A\ar[r]\ar[d]^f &\mc{Q}_{\mr{\acute{e}t}}/\overline{Y} \ar[r]\ar[d]^{f_{\mr{\acute{e}t,d}}} &0   \\
0\ar[r] &G'\ar[r] &A'\ar[r] &\mc{Q}'_{\mr{\acute{e}t}}/\overline{Y}'\ar[r] &0 
}
\end{equation}
with exact rows on $\Set$, and a canonical commutative diagram 
\begin{equation}\label{eq2.2}
\xymatrix{
0\ar[r] &G\ar[r]\ar[d]^{f_{\mr{c}}} &A\ar[r]\ar[d]^f &\mc{Q}/\overline{Y}\ar[r]\ar[d]^{f_{\mr{d}}} &0   \\
0\ar[r] &G'\ar[r] &A'\ar[r] &\mc{Q}'/\overline{Y}'\ar[r] &0 
}
\end{equation}
with exact rows on $\Skfl$. And the homomorphism $f_{\mr{d}}$ is just the sheafification of $f_{\mr{\acute{e}t,d}}$. We call $f_{\mr{c}}$ the connected (or semi-abelian) part of $f$, and $f_{\mr{\acute{e}t,d}}$ (resp. $f_{\mr{d}}$) the discrete part of $f$ on $\Set$ (resp. $\Skfl$).
\end{lem}
\begin{proof}
By \cite[9.2]{k-k-n2}, we have $\mr{Hom}_{\Set}(G,\mc{Q}'_{\mr{\acute{e}t}}/\overline{Y}')=0$. Then diagram (\ref{eq2.1}) follows. To achieve diagram (\ref{eq2.2}), it suffices to show that $\mr{Hom}_{\Skfl}(G,\mc{Q}'/\overline{Y}')=0$. By Lemma \ref{lem3} and \cite[Lem. 2.3]{zha3}, we have
$$\delta_*(\mc{Q}'/\overline{Y}')=\delta_*\mc{Q}'/\overline{Y}'=(\mc{Q}'_{\mr{\acute{e}t}}\otimes_{\Z}\Q)/\overline{Y}'.$$
It follows then 
$$\mr{Hom}_{\Skfl}(G,\mc{Q}'/\overline{Y}')=\mr{Hom}_{\Set}(G,\delta_*(\mc{Q}'/\overline{Y}'))=\mr{Hom}_{\Set}(G,(\mc{Q}'_{\mr{\acute{e}t}}\otimes_{\Z}\Q)/\overline{Y}').$$
The argument of \cite[9.2]{k-k-n2} also implies $\mr{Hom}_{\Set}(G,(\mc{Q}'_{\mr{\acute{e}t}}\otimes_{\Z}\Q)/\overline{Y}')=0$. Hence we get $\mr{Hom}_{\Skfl}(G,\mc{Q}'/\overline{Y}')=0$, which finishes the proof.

\end{proof}

\begin{prop}\label{prop2.1}
Suppose that the underlying scheme of $S$ is connected. Let $f:A\rightarrow A'$ be a homomorphism of log abelian varieties over $S$. Consider the following statements.
\begin{enumerate}[(1)]
\item $f$ is an isogeny.
\item $\mr{ker}f\in (\mr{fin}/S)_{\mr{f}}$ and $f$ is surjective.
\item $\mr{ker}f\in (\mr{fin}/S)_{\mr{r}}$ and $\mr{dim}A=\mr{dim}A'$.
\item $f$ is surjective and $\mr{dim}A=\mr{dim}A'$.
\end{enumerate}
Then we have $(1)\Leftrightarrow(2)\Leftrightarrow(3)\Rightarrow (4)$.
\end{prop}
\begin{proof}
Apparently we have $(1)\Rightarrow (2)$. Conversely, $\mr{ker}f\in (\mr{fin}/S)_{\mr{f}}$ implies that $\mr{ker}f=\mr{ker}(A[n]\xrightarrow{f} A'[n])$ for some positive integer $n$. Then we get $\mr{ker}f\in (\mr{fin}/S)_{\mr{r}}$ by \cite[Prop. 18.1. (1)]{k-k-n4}. Hence we get $(2)\Rightarrow (1)$.

Now we suppose that $f$ is an isogeny. Then the base change $f_{\bar{s}}$ of $f$ to $(\mr{fs}/\bar{s})$ is an isogeny for any $s\in S$. Let $G$ (resp. $G'$) be the semi-abelian part of $A$ (resp. $A'$), and let $f_{\mr{c}}:G\rightarrow G'$ be the connected part of $f$. By \cite[Prop. 3.3]{zha3}, $f_{\mr{c},\bar{s}}$ is an isogeny for any $s\in S$. Then $f_{\mr{c}}$ is flat by fibral flatness criterion, and quasi-finite. Hence we have $\mathop{\mr{dim}}A=\mathop{\mr{dim}}G=\mathop{\mr{dim}}G'=\mathop{\mr{dim}}A'$. This shows $(3)\Leftarrow (1)\Rightarrow (4)$.

Lastly, we show $(3)\Rightarrow (1)$. By \cite[prop. 3.3]{zha3}, $f_{\mr{c},\bar{s}}$ is an isogeny for any $s\in S$. Then we have that $f_{\mr{c}}$ is faithfully flat, hence it is surjective. Let $f_{\mr{d}}$ be the discrete part of $f$. To show the surjectivity of $f$, it suffices to show the surjectivity of $f_{\mr{d}}$ by the diagram (\ref{eq2.2}). The surjectivity of $f_{\mr{d}}$ is a local problem, hence we may assume that the underlying scheme of $S$ is noetherian and the discrete parts of $A$ and $A'$ come from admissible pairings
$$\langle,\rangle:X\times Y\rightarrow (\Gml/\Gm)_{\Set}$$ and 
$$\langle,\rangle':X'\times Y'\rightarrow (\Gml/\Gm)_{\Set}$$
respectively. By \cite[Thm. 7.6]{k-k-n2}, the homomorphism $f_{\mr{\acute{e}t,d}}$ comes from a pair $(u,v)$ of homomorphisms $u:\overline{X}'\rightarrow\overline{X}$, $v:\overline{Y}\rightarrow\overline{Y}'$ such that $\langle u(x'),y\rangle=\langle x',v(y)\rangle'$ for any $x'\in\overline{X}', y\in\overline{Y}$, where $\overline{X}$, $\overline{Y}$, $\overline{X}'$ and $\overline{Y}'$ are as in \cite[Thm. 7.6]{k-k-n2}. Since $f_{\mr{d}}$ is the sheafification of $f_{\mr{\acute{e}t,d}}$, it also comes from $(u,v)$. For any $s\in S$, $f_{\mr{c},\bar{s}}$ being isogeny implies that $u_{\bar{s}}$ is injective of finite cokernel. Since both $\overline{X}$ and $\overline{X}'$ are constructible sheaves on the small \'etale site of $S$, and the underlying scheme of $S$ is noetherian, the homomorphism $u$ has to be injective with cokernel killed by some positive integer. By \cite[Lem. 2.3]{zha3}, we have $\delta_*\mc{Q}=\mc{Q}_{\mr{\acute{e}t}}\otimes_{\Z}\Q$ and $\delta_*\mc{Q}'=\mc{Q}'_{\mr{\acute{e}t}}\otimes_{\Z}\Q$. It follows that $\mc{Q}\rightarrow\mc{Q}'$ is an isomorphism, hence $f_{\mr{d}}$ is surjective. This finishes the proof.
\end{proof}

\begin{cor}\label{cor2.1}
Let $A$ be a log abelian variety over $S$, $n$ a positive integer. Then the multiplication by $n$ map $n_A:A\rightarrow A$ is an isogeny.
\end{cor}
\begin{proof}
The kernel $A[n]$ of $n_A$ lies in $(\mr{fin}/S)_{\mr{r}}$ by \cite[Prop. 18.1. (1)]{k-k-n4}. Hence $n_A$ is an isogeny by the equivalence $(1)\Leftrightarrow(3)$ of Proposition \ref{prop2.1}.
\end{proof}

\section{Extending finite subgroup schemes of semi-stable abelian varieties}
From now on, we assume that $S=\mathop{\mr{Spec}}R$ with $R$ a complete discrete valuation ring. Let $K$ be the fraction field of $R$, $\pi$ a chosen uniformiser, and $k$ the residue field of $R$. We regard $S$ as an fs log scheme with respect to the canonical log structure, i.e. the log structure associated to $\N\rightarrow R,1\mapsto \pi$.  Let $\mr{LAV}_S$ be the category of log abelian varieties over $S$, and $\mr{AV}_K^{\mr{SSt}}$ the category of abelian varieties with semi-stable reduction over $K$. Taking the generic fiber gives rise to a functor $(-)_K:\mr{LAV}_S\rightarrow \mr{AV}_K^{\mr{SSt}}$.

We have the following beautiful theorem by \cite[13.4]{k-k-n4}.

\begin{thm}\label{thm3.1}
The functor $$(-)_K:\mr{LAV}_S\rightarrow \mr{AV}_K^{\mr{SSt}}$$
is an equivalence of categories.
\end{thm}

\begin{rmk}\label{rmk3.1}
Theorem \ref{thm3.1} is a special case of one of the equivalences of categories stated in \cite[13.4]{k-k-n4}. The authors will give the proof in a sequel of their series of papers on log abelian varieties. A proof for the split totally degenerate case of Theorem \ref{thm3.1} can be found in \cite{zha1}, and the inverse functor to $(-)_K$ is called the degeneration functor there. 
\end{rmk}

As mentioned at the beginning of \cite{kat3}, a finite subgroup scheme of $A_K\in \mr{AV}_K^{\mr{SSt}}$ does not necessarily extend to a finite flat group scheme over $S$. However, we can make the impossible possible if we use log finite flat group schemes instead of finite flat group schemes.

\begin{prop}\label{prop3.1}
Let $f_K:A_K\rightarrow A_K'$ be an isogeny in $\mr{AV}_K^{\mr{SSt}}$. Then $f$ extends to an isogeny $f:A\rightarrow A'$ of log abelian varieties over $S$, where $A$ and $A'$ are the log abelian varieties corresponding to $A_K$ and $A_K'$ respectively under the equivalence of Theorem \ref{thm3.1}.
\end{prop} 
\begin{proof}
Since $f_K$ is an isogeny, there exists another isogeny $g_K:A_K'\rightarrow A_K$ such that $g_K\circ f_K=n_{A_K}$ and $f_K\circ g_K=n_{A_K'}$ for some positive integer $n$. By the equivalence of categories of Theorem \ref{thm3.1}, the homomorphism $f_K$ (resp. $g_K$) extends to a homomorphism $f:A\rightarrow A'$ (resp. $g:A'\rightarrow A$). And we still have $g\circ f=n_A$ and $f\circ g=n_{A'}$. The kernel-cokernel exact sequence gives rise to two exact sequences
$$ \mr{coker}(g) \rightarrow\mr{coker}(n_{A'}) \rightarrow \mr{coker}(f)\rightarrow 0 $$
and
$$0\rightarrow \mr{ker}(f) \rightarrow A[n] \rightarrow \mr{ker}(g)\rightarrow\mr{coker}(f) .$$
The homomorphisms $n_{A'}$ and $n_A$ are both isogenies by Corollary \ref{cor2.1}, hence we have $\mr{coker}(n_{A'})=\mr{coker}(n_{A})=0$ and $A[n],A'[n]\in (\mr{fin}/S)_{\mr{r}}$. Then we have $\mr{coker}(f)=0$ and a short exact sequence
\begin{equation}\label{eq3.1}
0\rightarrow \mr{ker}(f) \rightarrow A[n] \rightarrow \mr{ker}(g)\rightarrow 0.
\end{equation}
In order to show that $f$ is an isogeny, we are reduced to show $\mr{ker}(f)\in (\mr{fin}/S)_{\mr{f}}$ by Proposition \ref{prop2.1}. This is a local problem for the Kummer flat topology, hence we may assume $A[n],A'[n]\in (\mr{fin}/S)_{\mr{c}}$ without loss of generality. Since both $A[n]$ and $A'[n]$ are finite over $S$, the homomorphism $A[n]\xrightarrow{f} A'[n]$ is automatically finite. Hence $\mr{ker}(f)$ is a classical finite group scheme over $S$. We are left with showing that $\mr{ker}(f)$ is flat over $S$.

Suppose that $\mr{ker}(f)=\mathop{\mr{Spec}}B_1$, $A[n]=\mathop{\mr{Spec}}B_2$ and $\mr{ker}(g)=\mathop{\mr{Spec}}B_3$. It suffices to show that $B_1$ is a free $R$-module. Note that $\mr{ker}(g)$ is finite over $S$ by the same argument as for $\mr{ker}(f)$. Let $r_1$ (resp. $r_3$) be the $R$-rank of the free part of $B_1$ (resp. $B_3$), and $t_1$ (resp. $t_3$) the length of the torsion part of $B_1$ (resp. $B_3$). The $R$-module $B_2$ is free, let $r_2$ be the $R$-rank of $B_2$. The base change of the short exact sequence (\ref{eq3.1}) to the closed point of $S$ remains exact, hence we get $r_2=(r_1+t_1)+(r_3+t_3)$. The base change of the short exact sequence (\ref{eq3.1}) to the generic point of $S$ remains exact, hence we get $r_2= r_1+r_3$. It follows that $t_1=t_3=0$ and $B_1$ is a free module of rank $r_{1}$ over $R$. This finishes the proof.
\end{proof}

\begin{thm}\label{thm3.2}
Let $A_K$ be a semi-stable abelian variety over $K$, and $F_K$ a finite subgroup scheme of $A_K$. Let $A$ be the log abelian variety over $S$ extending $A_K$ guaranteed by Theorem \ref{thm3.1}. Then $F_K$ extends to a log finite flat group scheme $F$ over $S$, which is a subgroup sheaf of $A$. Moreover the quotient $A/F$ is a log abelian variety over $S$.
\end{thm}
\begin{proof}
Let $A_K'$ be the abelian variety $A_K/F_K$, then we get an isogeny $f_K:A_K\rightarrow A_K'$ of abelian varieties. By Proposition \ref{prop3.1}, $f_K$ extends to an isogeny $f:A\rightarrow A'$ of log abelian varieties over $S$. Then $F:=\mr{ker}(f)$ is a log finite flat group scheme over $S$ whose generic fiber is $F_K$. Apparently we have that the quotient $A/F$ is the log abelian variety $A'$.
\end{proof}

The following corollary is well-known. Using the theory of log abelian variety, we can give a new proof to this result.

\begin{cor}\label{cor3.1}
Let the notation be as in Theorem \ref{thm3.2}. Assume that the order of $F_K$ is invertible in $R$. Then $F_K$ is tamely ramified.
\end{cor}
\begin{proof}
By Theorem \ref{thm3.2}, $F_K$ extends to $F\in (\mr{fin}/S)_{\mr{r}}$. Let
$$0\rightarrow F^{\circ}\rightarrow F\rightarrow F^{\mr{\acute{e}t}}\rightarrow 0$$
be the connected-\'etale sequence of $F$, see \cite[Lem. A.1]{zha3}. Since the order of $F_K$ is invertible in $R$, the connected part $F^{\circ}$ is trivial. Hence $F=F^{\mr{\acute{e}t}}$ is locally constant for the Kummer flat topology. By \cite[Thm. A.2]{zha3}, $F$ is already locally constant for the Kummer \'etale topology. Hence there exists a finite Kummer \'etale cover $S':=\mathop{\mr{Spec}}R'$ of $S$ such that $F_{S'}:=F\times_SS'$ is constant, where $R'$ is a tamely ramified extension of $R$ and $S'$ is endowed with the canonical log structure. It follows that $F_K$ is tamely ramified.
\end{proof}

\appendix
\section{Kummer flat covers are open}
The theory of Kummer flat topology is developed by Kazuya Kato in his preprint \cite{kat2}. We refer to the published papers \cite{niz1} and \cite{i-n-t1} for Kummer flat topology. Unfortunately to the author's knowledge, the result of Lemma \ref{lemA.1}, which we have used in the proof of Lemma \ref{lem2}, is only available in the preprint \cite{kat2}. For this reason, we present a proof of Lemma \ref{lemA.1} by modifying the proof of \cite[Cor. 3.7]{ill1}, which concerns Kummer \'etale covers.

\begin{lem}\label{lemA.1}
Let $f:X\rightarrow Y$ be a map of fs log schemes. Assume that $f$ is Kummer log flat, see \cite[Def. 2.13]{niz1}, and the underlying map of schemes of $f$ is locally of finite presentation. Then the underlying map of schemes of $f$ is open.
\end{lem}
\begin{proof}
The property of being open for a map of schemes is local for the classical flat topology. Hence without loss of generality, we may assume that the Kummer log flat morphism $f$ admits a global chart $(P\rightarrow M_X,Q\rightarrow M_Y,Q\xrightarrow{u} P)$, such that:
\begin{enumerate}[(1)]
\item $u$ is a homomorphism of fs monoids of Kummer type, with $u^{\mr{gp}}:Q^{\mr{gp}}\rightarrow P^{\mr{gp}}$ injective;
\item the morphism $\alpha$ in the following commutative diagram
$$\xymatrix{
X\ar@/_1pc/[rdd]_f\ar[rd]^{\alpha}\ar@/^1pc/[rrd]  \\
&Y\times_{\Spec\Z[Q]}\Spec\Z[P]\ar[r]\ar[d]^{\beta}  &\Spec\Z[P]\ar[d]^{\gamma}  \\
&Y\ar[r] &\Spec\Z[Q]
}$$
induced by the chart, is classically flat.
\end{enumerate}
Since $\gamma$ is clearly of finite presentation, $\beta$ is of finite presentation. Since $f$ is locally of finite presentation, we have that $\alpha$ is locally of finite presentation by \cite[\href{http://stacks.math.columbia.edu/tag/02FV}{Tag 02FV}]{stacks-project}. It follows that $\alpha$ is an open map of schemes by \cite[\href{http://stacks.math.columbia.edu/tag/01UA}{Tag 01UA}]{stacks-project}. To show that $f$ is open, it suffices to show that $\beta$ is open. 

The morphism $\gamma$ is clearly a finite map, hence so is $\beta$. It follows that $\beta$ is a closed map. Since $\Z[Q]\rightarrow\Z[P]$ is both injective and finite, the map $\gamma$ is surjective. It follows that $\beta$ is surjective. Hence the closed surjective map $\beta$ realizes the topology of $Y$ as the quotient topology of the topology of $X':=Y\times_{\Spec\Z[Q]}\Spec\Z[P]$. Let $U$ be an open subset of $X'$. To show that $\beta(U)$ is open, it suffices to show that $\beta^{-1}(\beta(U))$ is open.

Let $G$ be the group scheme $\Spec\Z[P^{\mr{gp}}/Q^{\mr{gp}}]$ over $\Spec\Z$ endowed with the trivial log structure. The group scheme $G$ acts on $X'$ over $Y$ by
\begin{align*}
\mc{O}_{X'}=\mc{O}_Y\otimes_{\Z[Q]}\Z[P]&\rightarrow\mc{O}_{G\times_{\Spec\Z}X'}=\mc{O}_Y\otimes_{\Z[Q]}\Z[ P^{\mr{gp}}/Q^{\mr{gp}}\oplus P]  \\
1\otimes a & \mapsto 1\otimes(\text{$a$ mod $Q^{\mr{gp}}$},a) \text{, for $a\in P$}.
\end{align*}
Let $(P\oplus_QP)^{\mr{fs}}$ denote the amalgamated sum of $P\xleftarrow{u}Q\xrightarrow{u}P$ in the category of fs monoids, then we have 
\begin{equation}\label{eqA.1}
(P\oplus_QP)^{\mr{fs}}\xrightarrow{\cong} P^{\mr{gp}}/Q^{\mr{gp}}\oplus P, (a,b)\mapsto (\bar{b},a+b)
\end{equation}
by \cite[Lem. 3.3]{ill1}. The isomorphism (\ref{eqA.1}) induces an isomorphism 
\begin{equation}\label{eqA.2}
G\times_{\Spec\Z}X'\xrightarrow{\cong}X'\times_YX',
\end{equation}
which is nothing but the map $(g,x)\mapsto (x,gx)$. By Nakayama's fourth point Lemma, see \cite[Lem. 2.2]{ill1}, we have $\beta^{-1}(\beta(U))=\mr{pr}_2(\mr{pr}_1^{-1}(U))$, where $\mr{pr}_1$ (resp. $\mr{pr}_2$) denotes the first (resp. second) projection of $X'\times_YX'$ to $X'$. Taking the isomorphism \ref{eqA.2} into account, $\beta^{-1}(\beta(U))$ is further identified with the image of $G\times_{\Spec\Z}U$ under the action morphism $G\times_{\Spec\Z}X'\rightarrow X'$. Since the structure morphism of the group scheme $G$ over $\Spec\Z$ is a classical flat cover, the projection to the second factor $G\times_{\Spec\Z}X'\rightarrow X'$ is also a classical flat cover. Since 
$$G\times_{\Spec\Z}X'\xrightarrow{\cong}G\times_{\Spec\Z}X',(g,x)\mapsto(g,gx),$$
the action morphism $G\times_{\Spec\Z}X'\rightarrow X'$ is also a classical flat cover. In particular, the action morphism $G\times_{\Spec\Z}X'\rightarrow X'$ is open, hence $\beta^{-1}(\beta(U))$ is open. This finishes the proof.
\end{proof}

\section*{Acknowledgement}
I thank Professor Chikara Nakayama for very helpful correspondences, as well as for his constant encouragements. I thank Professor Ulrich G\"ortz for very helpful discussion. I am grateful to the anonymous referee for his comments and suggestions, in particular for pointing out a gap in the proof of Lemma \ref{lem2}. This work has been supported by SFB/TR 45 ``Periods, moduli spaces and arithmetic of algebraic varieties''.

\bibliographystyle{alpha}
\bibliography{bib}

\end{document}